\documentclass[a4paper, 10pt, reqno]{amsart}

\usepackage{amsthm}
\usepackage{amsmath}
\usepackage{amsfonts}
\usepackage{amssymb}
\usepackage{latexsym}

\usepackage{graphicx}

\usepackage{tabularx}
\usepackage[margin=1.2in]{geometry}
\usepackage{enumitem}
\usepackage{ae}
\usepackage[T1]{fontenc}

\newtheorem{teo}{Theorem}
\newtheorem{cor}{Corollary}

\newtheorem{rem}{Remark}
\newtheorem{prop}{Proposition}
\newtheorem{lem}{Lemma}

\newcommand{\R}{\mathbb{R}}

\begin{document}
\title[Asympt. analysis of radial sign-changing solutions of the B.-N. problem]{Asymptotic analysis for radial sign-changing solutions of the Brezis-Nirenberg problem in low dimensions}
\author{Alessandro Iacopetti and Filomena Pacella}
\date{}
\thanks{Research partially supported by MIUR-PRIN project-201274FYK7\underline\ 005 and GNAMPA-INDAM.}
\address[Alessandro Iacopetti]{Dipartimento di Matematica e Fisica, Universit\'a degli Studi di Roma Tre, L.go S. Leonardo Murialdo 1, 00146 Roma, Italy}
\email{iacopetti@mat.uniroma3.it}
\address[Filomena Pacella]{Dipartimento di Matematica, Universit\'a di Roma "La Sapienza", P.le Aldo Moro 5, 00185 Roma, Italy}
\email{pacella@mat.uniroma1.it}

\begin{abstract}
We consider the classical Brezis-Nirenberg problem in the unit ball of $\R^N$, $N\geq 3$ and analyze the asymptotic behavior of nodal radial solutions in the low dimensions $N=3,4,5,6$ as the parameter converges to some limit value which naturally arises from the study of the associated ordinary differential equation.
%
%
\end{abstract}

\maketitle

\section{Introduction}
We consider the Brezis-Nirenberg problem

\begin{equation} \label{PBN}
\begin{cases}
-\Delta u = \lambda u + |u|^{2^* -2}u & \hbox{in}\ B_1\\
u=0 & \hbox{on}\ \partial B_1,
\end{cases}
\end{equation}
where $\lambda >0$, $2^*=\frac{2N}{N-2}$ and $B_1$ is the unit ball of $\R^N$, $N\geq 3$.

The aim of the paper is to get asymptotic results for radial sign-changing solutions $u_\lambda$ of \eqref{PBN} in dimensions $N=3,4,5,6$. This will give the asymptotic profile of the positive and negative part of $u_\lambda$ as $\lambda$ tends to some limit value.

To motivate our analysis and to explain our results we need to recall a few known results.

The first fundamental results about the existence of positive solutions were obtained by H. Brezis and L. Nirenberg in 1983 in the celebrated paper \cite{BN}. From their results it came out that the dimension was going to play a crucial role in the study of \eqref{PBN} in a general bounded domain $\Omega$.
Indeed they proved that if $N\geq4$ there exists a positive solution of \eqref{PBN} for every $\lambda \in (0,\lambda_1(\Omega))$, $\lambda_1(\Omega)$ being the first eigenvalue of $-\Delta$ in $\Omega$ with Dirichlet boundary conditions, while if $N=3$ positive solutions exist only for $\lambda$ away from zero.

Since then several other interesting results were obtained for positive solutions, in particular about the asymptotic behavior of solutions, mainly for $N\geq 5$, because also the case $N=4$ presents more difficulties compared to the higher dimensional ones.

Concerning the case of sign-changing solutions, existence results hold if $N\geq 4$ both for $\lambda \in (0,\lambda_1(\Omega))$ and $\lambda > \lambda_1(\Omega)$ as shown in \cite{ABP2}, \cite{CFP}, \cite{CW}.

The case $N=3$ presents even more difficulties than in the study of positive solutions. In particular in the case of the ball is not yet known what is the least value $\bar\lambda$ of the parameter $\lambda$ for which sign-changing solutions exist, neither whether $\bar\lambda$ is larger or smaller than $\lambda_1(B_1)/4$. This question, posed by H. Brezis, has been given a partial answer in \cite{AMP3}.

However it is interesting to observe that in the study of sign-changing solutions even the "low dimensions" $N=4,5,6$ exhibit some peculiarities. Indeed it was first proved by Atkinson, Brezis and Peletier in \cite{ABP} and \cite{ABP2} that if $\Omega$ is the ball $B_1$ there exists $\lambda^*=\lambda^*(N)$ such that there are no radial sign-changing solutions of \eqref{PBN} for $\lambda \in (0,\lambda^*)$. Later this result was proved in \cite{AY2} in a different way.

Moreover, for $N\geq 7$ a recent result of Schechter and Zou \cite{SZ} shows that in any bounded smooth domain there exist infinitely many sign-changing solutions for any $\lambda >0$. Instead if $N=4,5,6$ only $N+1$ pairs of  solutions, for all $\lambda>0$, have been proved to exist in \cite{CW} but it is not clear that they change sign.

Coming back to radial sign-changing solutions and to the question of existence or nonexistence of them, according to the dimension, as shown by Atkinson, Brezis and Peletier, it is interesting to understand in which way these results can be extended to other bounded domains and to which kind of solutions.

In order to analyze this question let us divide the discussion in two cases: the first one when the dimension $N$ is greater or equal than $7$ and the second one when $N<7$.

In the first case ($N\geq 7$) radial sign-changing solutions $u_\lambda$ exist for all $\lambda>0$, if the domain is a ball, and analyzing the asymptotic behavior of those of least energy, as $\lambda \rightarrow 0$, it is proved in \cite{Iac} that their limit profile is that of a "tower of two bubbles". This terminology means that the positive part and the negative part of the solutions $u_\lambda$ concentrate at the same point (which is obviously the center of the ball) as $\lambda\rightarrow 0$ and each one has the limit profile, after suitable rescaling, of a "standard bubble" in $\R^N$, i.e. of a positive solution of the critical exponent problem in $\R^N$. More precisely the solutions can be written in the following way:
\begin{equation}\label{shapesol}
u_\lambda=PU_{\delta_1,\xi}-PU_{\delta_2,\xi}+w_\lambda,
\end{equation}
where $PU_{\delta_i,\xi}$, $i=1,2$ is the projection on $H_0^1(\Omega)$ of the regular positive solution of the critical problem in $\R^N$, centered at $\xi=0$, with rescaling parameter $\delta_i$ and $w_\lambda$ is a remainder term which converges to zero in $H_0^1(\Omega)$ as $\lambda\rightarrow 0$.

Inspired by this result one could then search for solutions of type \eqref{shapesol} in general bounded domains since these kind of solutions can be viewed as the ones which play the same role of the radial solutions in the case of the ball. This has been done recently in \cite{IACVAIR}, where solutions of the type \eqref{shapesol} have been constructed for $\lambda$ close to zero in some symmetric bounded domains (the symmetry makes their construction a bit easier, but the same result should be true in any bounded domain).

On the contrary, coming to the case $N<7$, in view of the nonexistence result of nodal radial solutions of \cite{ABP2} it is natural to conjecture that, in general bounded domains, there should not be solutions of the form \eqref{shapesol} for $\lambda$ close to zero. Indeed this has been recently proved in \cite{IacPac} if $N=4,5,6$, the case $N=3$ being obvious.

On the other side, if $N<7$, radial nodal solutions exist for $\lambda$ bigger than a certain value $\bar\lambda_2$ which can be studied by analyzing the associated ordinary differential equation (see \cite{ABP2}, \cite{AG}, \cite{GG2}).

Therefore, to the aim of getting analogous existence results in other bounded domains, the first step would be to analyze the asymptotic behavior of nodal radial solutions in the ball, for $\lambda \rightarrow \bar \lambda_2$, in order to understand their limit profile and guess what kind of solutions one can construct in other domains, and for which values of the parameter $\lambda$.

This is the subject of our paper.

Denoting by $u_\lambda$ a nodal radial solutions of \eqref{PBN} having two nodal regions and such that $u_\lambda(0)>0$ we get the following results:
\begin{itemize}
\item[\textbf{(i)}:] if $N=6$ then $\bar\lambda_2 \in (0,\lambda_1(B_1))$, $\lambda_1(B_1)$ being the first eigenvalue of $-\Delta$ in $H_0^1(B_1)$, and we have that, as $\lambda \rightarrow \bar\lambda_2$, $u_\lambda^+$ concentrate at the center of the ball, $\|u_\lambda^+\|_\infty\rightarrow + \infty$, and a suitable rescaling of $u_\lambda^+$ converges to the standard  positive solution of the critical problem in $\R^N$. Instead $u_\lambda^-$ converges to the unique positive solution of \eqref{PBN} in $B_1$, as $\lambda\rightarrow \bar \lambda_2$;
\item[\textbf{(ii)}:] if $N=4,5$ then $\bar\lambda_2=\lambda_1(B_1)$ and $u_\lambda^+$ behaves as for the case $N=6$, while $u_\lambda^-$ converges to zero uniformly in $B_1$;
\item[\textbf{(iii)}:] if $N=3$ then $\bar\lambda_2=\frac{9}{4}\lambda_1(B_1)$ and $u_\lambda^+$ behaves as for the case $N=6$, while $u_\lambda^-$ converges to zero uniformly in $B_1$.
\end{itemize}

In view of these results we conjecture that, in general bounded domains $\Omega$, for some "limit value" $\bar\lambda_2=\bar\lambda_2(N,\Omega)$ there should exist solutions with similar asymptotic profile as $\lambda \rightarrow \bar\lambda_2$. The number $\bar\lambda_2$ should be $\lambda_1(\Omega)$ in dimension $N=4,5$. Some work in this direction is in progress.

The paper is divided in three sections. In Section 2 we mainly recall some preliminary results. In Section 3 we analyze the asymptotic behavior of the positive part of the solutions, for all dimensions $N=3,4,5,6$. In Section 4 we analyze the negative part in the case $N=6$ and in Section 5 we complete the cases $N=3,4,5$.

%
%
%

\section{Some preliminary results}

If $u_\lambda$ is a radial sign-changing solution of (\ref{PBN}) then we can write $u_\lambda=u_\lambda(r)$, where $r=|x|$ and $u_\lambda(r)$ is a solution of the problem
\begin{equation}\label{ODE}
\begin{cases}
u_\lambda^{\prime \prime} +\frac{n-1}{r}u_\lambda^{\prime} + \lambda u_\lambda + |u_\lambda |^{2^* -2}u_\lambda =0,  & \hbox{in}\ \ (0,1),\\
u_\lambda^\prime( 0)=0, \ \ u_\lambda(1)=0.&
\end{cases}
\end{equation}
We consider the following transformation

\begin{equation}\label{tran}
r \mapsto \left(\frac{N-2}{\sqrt{\lambda}\ r}\right)^{N-2},\\
\ \ u_\lambda \mapsto y(t):=\lambda^{-1/(p-1)}u_\lambda\left(\frac{N-2}{\sqrt{\lambda}\ t^{\frac{1}{N-2}}}\right).
 \end{equation} 
 
It is elementary to see that since $u_\lambda$ is a solution of the differential equation in (\ref{ODE}) then $y=y(t)$ solves
\begin{equation}\label{EF}
y^{\prime\prime}+t^{-k}(y+|y|^{p-1}y)=0,
\end{equation}
 in the interval $\left(\left(\frac{N-2}{\sqrt{\lambda}}\right)^{N-2},+\infty\right)$, where $k:=2 \frac{N-1}{N-2}$. It is clear that the transformation (\ref{tran}) generates a one-to-one correspondence between solutions of the differential equation in (\ref{ODE}) and solutions of (\ref{EF}). Equation (\ref{EF}) is an Emden-Fowler type equation and since $k>2$ it is well known that, for any $\gamma \in \R$ the problem
\begin{equation}\label{EFS}
\begin{cases}
y^{\prime\prime}+t^{-k}f(y)=0,  & \hbox{in}\ \ (0,+\infty),\\
y(t) \rightarrow \gamma,& \hbox{as}\ t \rightarrow + \infty,
\end{cases}
\end{equation}
 where $f(y):=y+|y|^{p-1}y$, has a unique solution defined in the whole $\R^+$ which we denote by $y(t;\gamma)$. Let us recall some results on the functions $y(t;\gamma)$ which are proved in \cite{ABP2}.
 
\begin{lem}\label{lem1}
Let $y=y(t,\gamma)$ be a solution of Problem (\ref{EFS}), then:
\begin{itemize}
\item[(a)] $y$ is oscillatory near $t=0$;
\item[(b)] the set $\{|y(\bar t)|;\ \bar t \ \hbox{extremum point of} \ y\}$ is an increasing sequence with respect to t; 
\item[(c)]  the set $\{|y^\prime(t_0)|;\  t_0 \ \hbox{zero of} \ y\}$ is a decreasing sequence with respect to t. 
\end{itemize}
\end{lem}

\begin{proof}
See Lemma 1 in \cite{ABP2}.
\end{proof}
\begin{lem}\label{lem2}
Let $y=y(t,\gamma)$ be a solution of Problem (\ref{EFS}) and let $T>0$ be one of its zeros, then 
$$|y(t)|< |y^\prime(T)| (T-t),$$
for all $0<t<T$.
\end{lem}
\begin{proof}
See Lemma 2 in \cite{ABP2}.
\end{proof}

We shall denote the sequence of zeros of $y(t;\gamma)$ by $T_n(\gamma)$, ordered backwards, precisely: 
$$\cdots < T_3(\gamma) < T_2 (\gamma) < T_1(\gamma) < + \infty.$$
We recall some results on the asymptotic behavior of the largest zero $T_1(\gamma)$ and on the slope $y^\prime(T_1(\gamma); \gamma)$ as $\gamma \rightarrow + \infty$.

\begin{lem}\label{lem3}
Let $y$ be a solution of Problem (\ref{EFS}) and $T_1(\gamma)$ its largest zero, then:
\begin{itemize}
\item[(a)] if $2<k<3$ (which corresponds to $N>4$), then
$$ T_1(\gamma)=A(k) \gamma^{6-2k} (1+o(1)) \ \ \hbox{as } \gamma \rightarrow +\infty,$$
where $A(k):=(k-1)^{\frac{k-3}{k-2}} \frac{\Gamma(3-k)/(k-2) \Gamma ((k-1)/(k-2))}{\Gamma(2/(k-2))}$, $\Gamma$ is the Gamma function.
\item[(b)] if $k=3$ (which corresponds to $N=4$), then
$$ T_1(\gamma)=2 \log \gamma (1+o(1))\ \ \hbox{as } \gamma \rightarrow +\infty; $$
\item[(c)] if $k=4$ (which corresponds to $N=3$), then there exists $\gamma_0 \in \R^+$ and two positive constants $A,B$ such that
$$ A<T_1(\gamma)<B \ \ \hbox{for all } \gamma\geq \gamma_0.$$
\end{itemize}
\end{lem}
\begin{proof}
The proof of (a), (b) is contained in \cite{ABP2}, Lemma 3 and the proof of (c) is contained in \cite{10}, Theorem 3.
\end{proof}

\begin{lem}\label{lem4}
For any $k>2$, let $y$ be a solution of Problem (\ref{EFS}) and $T_1(\gamma)$ its largest zero, then
$$y^\prime(T_1(\gamma))= (k-1)^{\frac{1}{k-2}} \gamma^{-1}(1+o(1)), \ \hbox{as} \ \gamma\rightarrow + \infty.$$
\end{lem}
\begin{proof}
See \cite{ABP2}, Lemma 4.
\end{proof}

To prove the existence of radial sign-changing solutions of (\ref{PBN}), with exactly two nodal regions, we consider the second zero $T_2(\gamma)$ of $y(t;\gamma)$. If we choose $\lambda=\lambda(\gamma)$ so that $T_2(\gamma)=\left(\frac{N-2}{\sqrt{\lambda}\ }\right)^{N-2}$, then the inverse transformation of (\ref{tran}) maps $t=T_2$ in $r=1$ and $y \mapsto u_\lambda$. Hence, for $\lambda=(N-2)^2 T_2(\gamma)^{- \frac{2}{N-2}}$, we obtain a function $u_\lambda$ which is a radial solution of (\ref{PBN}) having exactly two nodal regions; moreover $u_\lambda(0)=\lambda^{1/(p-1)} \gamma$. We observe also that thanks to the invertibility of (\ref{tran}) every radial sign-changing solution $u_\lambda$ of (\ref{PBN}) with two nodal regions corresponds to a solution $y=y(t;\gamma)$ of (\ref{EFS}) with $\gamma=\lambda^{- 1/(p-1)} u_\lambda(0)$, $T_2(\gamma)=\left(\frac{N-2}{\sqrt{\lambda}\ }\right)^{N-2}$.

We are interested in the study of the behavior of the map $\lambda_2: \R^+ \rightarrow \R^+$, defined by $\lambda_2(\gamma):=(N-2)^2 T_2(\gamma)^{- \frac{2}{N-2}}$. Clearly this map is continuos. In \cite{AG} (see Proposition 2 and Remark 4), it is proved that for $N=4$ it holds that $\lim_{\gamma \rightarrow 0} \lambda_2(\gamma)= \lambda_2(B_1)$, where $\lambda_2(B_1)$ is the second radial eigenvalue of $-\Delta$ in $H_0^1(B_1)$. Moreover the authors observe that this result holds for all dimensions $N\geq 3$. For the sake of completeness we give a complete proof of this fact. We begin with a preliminary lemma.

\begin{lem} \label{lemiac}
Let $u_\lambda$ be a radial solution of (\ref{PBN}), then we have $|u_\lambda (0)|=\|u_\lambda\|_{\infty}$.
\end{lem}
\begin{proof}
See \cite{Iac}, Proposition 2 or \cite{AG}, Lemma 8.
\end{proof}

\begin{prop}\label{proplamb}
Let $N\geq 3$ and $\lambda_2: \R^+ \rightarrow \R^+$ the function defined by $\lambda_2(\gamma):=(N-2)^2 T_2(\gamma)^{- \frac{2}{N-2}}$, where $T_2(\gamma)$ is the second zero of the function $y(t,\gamma)$, $y(t,\gamma)$ is the unique solution of (\ref{EFS}).
We have:
\begin{itemize}
\item[(a)] $\lambda_2(\gamma) < \lambda_2(B_1)$, for all $\gamma \in \R^+$;
\item[(b)] $\lim_{\gamma \rightarrow 0} \lambda_2(\gamma)= \lambda_2(B_1)$,
 \end{itemize}
  where $\lambda_2(B_1)$ is the second radial eigenvalue of $-\Delta$ in $H_0^1(B_1)$.
\end{prop}

\begin{proof}
To prove (a) we observe that (a) is equivalent to show that $T_2(\gamma) > \tau_2$ for all $\gamma \in \R^+$, where $\tau_2$ is the second zero of the function $\alpha: \R^+ \rightarrow \R$ defined by $\alpha(t):=A_\nu\sqrt{t}J_\nu(2\nu t^{- \frac{1}{2\nu}})$, where $A_\nu:=\nu^{-\nu}\Gamma(\nu+1)$, $\nu:=\frac{1}{k-2}=\frac{N-2}{2}$, is the first kind (regular) Bessel function of order $\nu$, namely
$J_\nu(s):=\sum_{j=0}^\infty \frac{(-1)^j}{\Gamma(j+1)\Gamma(j+\nu+1)}\left(\frac{s}{2}\right)^{\nu+2j}$. In fact, by a tedious computation, we see that $\alpha$ solves 
\begin{equation}\label{eqalpha}
\begin{cases}
\alpha^{\prime\prime}+ t^{-k} \alpha=0,& \hbox{in } (0,+\infty),\\
\alpha(t) \rightarrow 1, & \hbox{as }\ t\rightarrow + \infty.
\end{cases}
 \end{equation}
 Furthermore, let $\tau_2$ be the second zero of $\alpha$, then by elementary computations we see that the function $\varphi_2(x):=\alpha(\tau_2 |x|^{-(N-2)})$ solves
\begin{equation}
\begin{cases}
-\Delta \varphi_2 = \mu_2  \varphi_2   & \hbox{in}\ B_1\\
 \varphi_2=0 & \hbox{on}\ \partial B_1,
\end{cases}
\end{equation}
with $\mu_2=(N-2)^2 \tau_2^{- \frac{2}{N-2}}$. Clearly $\mu_2=\lambda_2(B_1)$. Hence $\lambda_2(\gamma) < \lambda_2(B_1)$ if and only if $T_2(\gamma) > \tau_2$. 

To show that $T_2(\gamma) > \tau_2$ for all $\gamma \in \R^+$ first observe that for all $\gamma \in \R^+$ we have $T_1(\gamma)>\tau_1$. In fact, setting $\lambda_1(\gamma):=(N-2)^2 T_1(\gamma)^{- \frac{2}{N-2}}$ as before we have that $\lambda_1(\gamma)< \lambda_1(B_1)$ if and only if $T_1(\gamma)>\tau_1$.  Since we know from \cite{BN} that equation (\ref{ODE}) has positive solutions only for $\lambda \in (0,\lambda_1(B_1))$ if $N\geq4 $, and only for $\lambda \in (\frac{\lambda_1(B_1)}{4},\lambda_1(B_1))$ if $N=3$, we deduce $T_1(\gamma)>\tau_1$ for all $\gamma \in \R^+$. 
Now we apply the Sturm's comparison theorem to the functions $y(t;\gamma)$, $\alpha(t)$, which are, respectively, solutions of the equations in (\ref{EFS}), (\ref{eqalpha}). To this end we write  $y^{\prime\prime}+t^{-k}q_2(t)y=0$ with $q_2(t):=1+|y|^{p-1}$ and since $\alpha^{\prime\prime}+ t^{-k} \alpha=0$ we set $q_1(t):\equiv 1$. Clearly $q_2(t)\geq q_1(t)$ for all $t>0$ (for all $\gamma \in \R^+$), thus $y$ is a Sturm majorant for $\alpha$, and applying the Sturm's comparison theorem in the interval $[\tau_2,\tau_1]$, since $T_1(\gamma)>\tau_1$ we deduce that $T_2(\gamma) \in (\tau_2,\tau_1)$. This concludes the proof of (a).

Let us prove (b). We consider $u_{\lambda_2(\gamma)}=u_{\lambda_2(\gamma)}(r)$ which is a solution of $(\ref{ODE})$ with exactly one zero in $(0,1)$, and $u_{\lambda_2(\gamma)}(0)=[\lambda_2(\gamma)]^{1/(p-1)}\gamma$. Setting $\varphi(x):=u_{\lambda_2}(|x|)$ it's clear that $\varphi$ is the second radial eigenfunction of
\begin{equation}
\begin{cases}
-\Delta \varphi = \lambda \varphi + |u_{\lambda_2(\gamma)}|^{2^* -2}\varphi & \hbox{in}\ B_1\\
\varphi=0 & \hbox{on}\ \partial B_1,
\end{cases}
\end{equation}
with eigenvalue $\lambda=\lambda_2(\gamma)$. Let us denote by $H_{0,rad}^1(B_1)$ the subspace of radially symmetric functions in $H_0^1(B_1)$. Thanks to the variational characterization of eigenvalues and Lemma \ref{lemiac} we have
\begin{equation}\label{eq1}
\begin{array}{lll}
\displaystyle \lambda_2(\gamma)&=& \displaystyle \min_{\substack{V \subset H_{0,rad}^1(B_1)\\ \dim V=2}} \max_{\substack{\varphi \in V\\ |\varphi|_2=1}}\left(\int_{B_1} |\nabla \varphi|^2 \ dx - \int_{B_1} |u_{\lambda_2(\gamma)}|^{2^* -2}\varphi^2 \ dx \right)\\
&>& \displaystyle \min_{\substack{V \subset H_{0,rad}^1(B_1)\\ \dim V=2}} \max_{\substack{\varphi \in V\\ |\varphi|_2=1}}\left(\int_{B_1} |\nabla \varphi|^2 \ dx - [\lambda_2(\gamma)]^{2/(p-1)}\gamma^2 \right)\\[16pt]
&=& \displaystyle \lambda_2(B_1) -  [\lambda_2(\gamma)]^{2/(p-1)}\gamma^2.
\end{array}
\end{equation}
Since $\lambda_2(\gamma)$ is bounded (because by (a) we have $\lambda_2(\gamma)<\lambda_2(B_1)$ and by definition $\lambda_2(\gamma)>0$), from (\ref{eq1}), we deduce that $\liminf_{\gamma \rightarrow 0} \lambda_2(\gamma)\geq \lambda_2(B_1)$. On the other hand, by the first step we get that $\limsup_{\gamma \rightarrow 0} \lambda_2(\gamma)\leq \lambda_2(B_1)$. Hence we deduce that $\lim_{\gamma \rightarrow 0} \lambda_2(\gamma)=\lambda_2(B_1)$ and the proof is concluded.
\end{proof}

More interesting is the behavior of $\lambda_2(\gamma)$ as $\gamma \rightarrow + \infty$. The next result that we recall shows how it strongly depends on the dimension $N$.

\begin{teo}\label{teolamb}
Let $\lambda_2:\R^+ \rightarrow \R^+$ be the function defined by $\lambda_2(\gamma):=(N-2)^2 T_2(\gamma)^{- \frac{2}{N-2}}$, where $T_2(\gamma)$ is the second zero of the function $y(t,\gamma)$, being $y(t,\gamma)$ is the unique solution (\ref{EFS}), and let $\lambda_1(B_1)$ be the first eigenvalue of $-\Delta$ in $H_0^1(B_1)$, then:
\begin{itemize}
\item[(a)]if $N\geq 7$ we have $\lim_{\gamma \rightarrow + \infty} \lambda_2(\gamma)=0$;
\item[(b)]if $N=6$ we have $\lim_{\gamma \rightarrow + \infty} \lambda_2(\gamma)=\lambda_0$,  for some $\lambda_0 \in (0,\lambda_1(B_1))$;
\item[(c)]if $N=4$ or $N=5$ we have $\lim_{\gamma \rightarrow + \infty} \lambda_2(\gamma)=\lambda_1(B_1)$;
\item[(d)]if $N=3$ we have $\lim_{\gamma \rightarrow + \infty} \lambda_2(\gamma)=\frac{9}{4}\lambda_1(B_1)=\frac{9}{4} \pi^2$.
\end{itemize}
\end{teo}
\begin{proof}
Statement (a) is a consequence of Theorem B in \cite{CSS}. Statements (b), (c) are proved in \cite{ABP2}, Theorem B. In Section 4 we give an alternative proof of (b). Statement (d) is proved in \cite{ABP4}.
\end{proof}

Let us define $\lambda_2^\star:=\inf \{\lambda_2(\gamma), \ \gamma \in \R^+ \}$. Gazzola and Grunau proved in \cite{GG2} that for $N=5$ it holds $\lim_{\gamma \rightarrow + \infty} \lambda_2(\gamma)=\lambda_1(B_1)^-$, in particular we deduce that for $N=5$ we have $\lambda_2^\star< \lambda_1(B_1)$ and hence $\lambda_2^\star=\lambda_2(\gamma_0)$ for some $\gamma_0 \in \R^+$. In the same paper it is also proved that for $N=4$ $\lim_{\gamma \rightarrow + \infty} \lambda_2(\gamma)=\lambda_1(B_1)^+$. Recently Arioli, Gazzola, Grunau, Sassone proved in \cite{AG} a stronger result: for $N=4$ we have $\lambda_2(\gamma) > \lambda_1(B_1)$ for all $\gamma \in \R^+$. Thus for $N=4$, we have $\lambda_2^\star = \lambda_1(B_1)$ and $\lambda_2^\star$ is not achieved. 

The asymptotic behavior of $\lambda_2(\gamma)$ as $\gamma \rightarrow + \infty$ for $N=6$ is still unknown. Nevertheless in Section 3 we give a characterization of the number $\lambda_0$ appearing in (b) of Theorem \ref{teolamb}. 

\section{Energy and asymptotic analysis of the positive part}

Let $u_{\lambda_2(\gamma)}$ the radial solution with exactly two nodal regions of (\ref{PBN}), for $\lambda=\lambda_2(\gamma)$, obtained in the previous section. To simplify the notation we omit the dependence on $\gamma$ and write $u_{\lambda_2}$. We recall that, by definition, for $\gamma \in \R^+$ we have $u_{\lambda_2}(0)>0$ and we denote by $r_{\lambda_2} \in (0,1)$ its node. 

 The aim of this section is to compute the limit energy of the positive part $u_{\lambda_2}^+$, as $\gamma \rightarrow +\infty$, as well as, to study the asymptotic behavior of a suitable rescaling of $u_{\lambda_2}^+$. 
We begin with recalling an elementary but crucial fact:

\begin{lem}\label{ELLEM} 
Let $u \in H_{rad}^1(B)$, where $B$ is a ball or an annulus centered at the origin of $\R^N$ and consider the rescaling $\tilde u (y):=M^{1/\beta} u(My)$, where $M>0$ is a constant, $\beta:=\frac{2}{N-2}$. We have:
\begin{itemize}
\item[(i):] $\|u\|_{B}^2=\|\tilde u\|_{M^{-1}B}^2$,
\item[(ii):]$|u|_{2^*,B}^{2^*}=|\tilde u|_{2^*, M^{-1}B}^{2^*}$,
\item[(iii):] $|u|_{2,B}^{2}=M^{2}|\tilde u|_{2, M^{-1}B}^{2}$.
\end{itemize}
\end{lem}
\begin{proof}
It suffices to apply the formula of change of variable for the integrals in (i), (ii), (iii). For the details see Lemma 2 in \cite{Iac}.
\end{proof}

In order to state the main result of this section we introduce some notation. We define the rescaled functions
 $$\tilde u_{\lambda_2}^+(y):=\frac{1}{M_{\lambda_2,+}}u_{\lambda_2}^+\left(\frac{y}{M_{\lambda_2,+}^\beta}\right),\ \ \ y \in B_{\sigma_{\lambda_2}}, $$ 
where $\beta:=\frac{2}{N-2}$, $\sigma_{\lambda_2}=M_{\lambda_2,+}^\beta {r_{\lambda_2}}$, $M_{\lambda_2,+}:=\|u_{\lambda_2}^+\|_{\infty,B_1}$. We observe that thanks to Lemma \ref{lemiac} and since $u_{\lambda_2}(0)>0$ we have $M_{\lambda_2,+}=\|u_{\lambda_2}\|_{\infty,B_1}=u_{\lambda_2}(0)$. The following theorem holds for all dimensions $N\geq3$, here we discuss the case $3 \leq N \leq 6$ (the case $N\geq7$ is studied in \cite{Iac}).

\begin{teo}\label{enestpp}
 Let $N=3,4,5,6$ and let $u_{\lambda_2}$ be the radial solution with exactly two nodal regions of (\ref{PBN}) with $\lambda=\lambda_2(\gamma)$ obtained in the previous section. Then
 \begin{itemize}
\item[(i):] $$J_{\lambda_2} (u_{\lambda_2}^+) \rightarrow \frac{1}{N}S^{N/2},$$
as $\gamma \rightarrow +\infty$, where $J_\lambda (u):=\frac{1}{2} \left( \int_{B_1}|\nabla u|^2 - \lambda |u|^2 \ dx\right)-\frac{1}{2^*} \int_{B_1} |u|^{2^*} \ dx$ is the energy functional related to (\ref{PBN}), $S$ is the best Sobolev constant for the embedding of $\mathcal{D}^{1,2}(\R^N)$ into $L^{2^*}(\R^N)$.
\item[(ii):] Up to a subsequence, the rescaled function $\tilde u_{\lambda}^+$ converges in $C^2_{loc}(\R^N)$ to $U_{0,\mu}$, as $\gamma \rightarrow + \infty$, where $U_{0,\mu}$ is the solution of the critical exponent problem in $\R^N$ centered at $x_0=0$ and with concentration parameter $\mu=\sqrt{N(N-2)}$. We recall that 
such functions are defined by
\begin{equation*}
U_{x_0,\mu}(x):=\frac{[N(N-2)\mu^2]^{(n-2)/4}}{[\mu^2+|x-x_0|^2]^{(n-2)/2}}.
\end{equation*}
\end{itemize}
\end{teo}

\begin{proof}
We start by proving (i).
Let $(u_{\lambda_2})$ be this family of solutions. Since $u_{\lambda_2}^+$ solves $-\Delta u = \lambda_2 u + u^p$ in $B_{r_{\lambda_2}}$ then, considering the rescaling $\hat u_{\lambda_2}^+(y):=r_{\lambda_2}^{1/\beta} u_{\lambda_2}^+(r_{\lambda_2} y)$, where $\beta:=\frac{2}{N-2}$, we see that $\tilde u_{\lambda_2}^+$ solves 

\begin{equation}\label{eq1lemest}
\begin{cases}
-\Delta u = \lambda_2 r_{\lambda_2}^{2} u + u^p & \hbox{in} \ B_1,\\
u>0 &\hbox{in} \ B_1,\\
u=0 & \hbox{on} \ \partial B_1.
\end{cases}
\end{equation}
Now we distinguish between two cases: $N=4,5,6$ and $N=3$. 

If $N=4,5,6$, then,  from to Lemma \ref{lem3} we deduce that $r_{\lambda_2}\rightarrow 0$ as $\gamma \rightarrow + \infty$, in particular this is true for $\lambda_2 r_{\lambda_2}^{2}$. From \cite{AYU} we know that $\hat u_{\lambda_2}^+$ is unique and it coincides with the solution found in \cite{BN}, which minimizes the energy $J_{\lambda_2r_{\lambda_2}^{2}}$; thus, since $\lambda_2 r_{\lambda_2}^{2} \rightarrow 0$ as $\gamma \rightarrow + \infty$ we get that $J_{\lambda_2 r_{\lambda_2}^{2}}(\hat u_{\lambda_2}^+) \rightarrow \frac{1}{N}S^{N/2}$. Thanks to Lemma \ref{ELLEM} we get that $J_{\lambda_2}( u_{\lambda_2}^+)=J_{\lambda_2 r_{\lambda_2}^{2}}(\hat u_{\lambda_2}^+) \rightarrow \frac{1}{N}S^{N/2}$ as $\gamma \rightarrow + \infty$.

Assume now that $N=3$. 
As stated in Lemma  \ref{lem3} we have that $r_{\lambda_2}$ is bounded away from zero. 
From a well know result of Brezis and Nirenberg (see \cite{BN}, Theorem 1) we have that \eqref{eq1lemest} has a  positive solution if and only if $ \lambda_2 r_{\lambda_2}^{2} \in (\frac{\pi^2}{4},\pi^2)$. As $\gamma \rightarrow +\infty$ we must have $\lambda_2 r_{\lambda_2}^{2} \rightarrow \frac{\pi^2}{4}$. Hence, the only possibility is that $J_{\lambda_2 r_{\lambda_2}^{2}}(\hat u_{\lambda_2}^+) \rightarrow \frac{1}{3}S^{3/2}$ as $\gamma \rightarrow + \infty$. As before thanks to Lemma \ref{ELLEM} we have $J_{\lambda_2}( u_{\lambda_2}^+)=J_{\lambda_2 r_{\lambda_2}^{1/\beta}}(\hat u_{\lambda_2}^+)$ and hence $J_{\lambda_2}( u_{\lambda_2}^+) \rightarrow \frac{1}{3}S^{3/2}$ as $\gamma \rightarrow + \infty$. The proof of (i) is complete.

We now prove (ii).
By definition the rescaled function $\tilde u_{\lambda_2}^+$ solves the following problem
\begin{equation}\label{probriscal}
\begin{cases}
-\Delta  u =\frac{\lambda_2}{M_{\lambda_2,+}^{2\beta}} u+   u^{2^* -1} & \hbox{in}\ B_{\sigma_{\lambda_2}},\\
 u > 0 & \hbox{in} \  B_{\sigma_{\lambda_2}},\\
 u =0 & \hbox{on} \ \partial B_{\sigma_{\lambda_2}},
\end{cases}
\end{equation}
where $\sigma_{\lambda_2}:=M_{\lambda_2}^\beta r_{\lambda_2}$.

Since the family $(\tilde u_{\lambda_2}^+)$ is uniformly bounded, then by standard elliptic theory we get that $\tilde u_{\lambda_2}^+ \rightarrow \tilde u $ in $C_{loc}^2(B_l)$, where $l$ is the limit of ${\sigma_{\lambda_2}}$ as $\gamma \rightarrow + \infty$. 
 We want to show that $$\lim_{\gamma\rightarrow + \infty} \sigma_{\lambda_2} = + \infty,$$
 so that the limit domain is the whole $\R^N$.
We can proceed in two different ways: one is to apply directly the estimates contained in Section 1,
the other one is to apply the methods of \cite{Iac}. We choose the second approach: arguing as in the proof of Proposition 9 in \cite{Iac}, taking into account that by (i) of Theorem \ref{enestpp}, $J_{\lambda_2}(u_{\lambda_2}^+) \rightarrow \frac{1}{N} S^{N/2}$, as $\gamma \rightarrow + \infty$, we see that up to a subsequence it cannot happen that $\lim_{\gamma\rightarrow + \infty} \sigma_{\lambda_2}$ is finite. 

Since $\frac{\lambda_2}{M_{\lambda_2,+}^{2\beta}}\rightarrow 0$, as $\gamma \rightarrow +\infty$, $\tilde u_{\lambda_2}^+$ converges in $C^2_{loc}(\R^N)$ to a positive solution $\tilde u$ of 
 \begin{equation*}
\begin{cases}
-\Delta  u =   |u|^{2^* -2}u & \hbox{in}\ \R^N\\
  u\rightarrow 0 &\ \hbox{as} \ |y|\rightarrow +\infty.
\end{cases}
\end{equation*}
Observe that this holds even in the case $N=3$, in fact by definition and Remark \ref{remerkel} we have $$\frac{\lambda_1(B_{\sigma_{\lambda_2}})}{4}=\frac{\pi^2}{4M_{\lambda_2,+}^{4}r_{\lambda_2}^2}=\frac{9}{4}\pi^2 (1+o(1)) \frac{1}{M_{\lambda_2,+}^{4}}=\frac{\lambda_2}{M_{\lambda_2,+}^{4}}(1+o(1))\rightarrow 0,$$ as $\gamma \rightarrow +\infty$. 

Since $\tilde u$ is radial and $\tilde u(0)=1$ then $\tilde u = U_{0,\mu}$ where $\mu=\sqrt{N(N-2)}$ (see Proposition 2.2 in \cite{CSS}). The proof is complete.
\end{proof}

\begin{rem}\label{remerkel}
We observe that for $N=3$, since $\lambda_2 r_{\lambda_2}^{2} \rightarrow \frac{\pi^2}{4}$  and (d) of Theorem \ref{teolamb} holds then we deduce that $r_{\lambda_2} \rightarrow \frac{1}{3}$.
\end{rem}

\section{Asymptotic analysis of the negative part in dimension $N=6$}
In this section we focus on the case $N=6$ which means to take $k=5/2$ in (\ref{EFS}).
As in \cite{ABP2} we define 

\begin{equation}\label{deftzero}
\begin{array}{lll}
t_0(\gamma)&:=&\inf \{t \in (0,+\infty);\  y^\prime >0\ \hbox{on } (t,+\infty)\},\\
y_0(\gamma)&:=&y(t_0(\gamma);\gamma).
\end{array}
\end{equation}
We have the following:
\begin{prop}\label{proptran}
Assume $k=5/2$. Then
\begin{itemize}
\item[(a)] $y_0(\gamma)=- \frac{1}{2} (1+o(1))$, as $\gamma \rightarrow +\infty$;
\item[(b)] $t_0(\gamma)= (\frac{2}{9} \gamma)^{2/3} (1+o(1))$, as $\gamma \rightarrow +\infty$.
\end{itemize}
\end{prop}
\begin{proof}
See \cite{ABP2}, Theorem 2.
\end{proof}
Let $u_\lambda$ be any radial solution of (\ref{PBN}) with exactly two nodal regions and without loss of generality  assume that $u_\lambda(0)>0$. We denote by 
$s_\lambda$ the global minimum point of $u_\lambda$. As in the previous section we set $M_{\lambda,+}:=\|u_\lambda^+\|_\infty$, $M_{\lambda,-}:=\|u_\lambda^-\|_\infty$, where $u_\lambda^+$, $u_\lambda^-$ are respectively the positive and the negative part of $u_\lambda$. 
Clearly, by definition, we have $u_\lambda^-(s_\lambda)=M_{\lambda,-}$. In order to estimate the energy of such solutions we need the following preliminary result.

\begin{prop} \label{propmlamb}
Let $N=6$ and let $(u_\lambda)$ be any family of radial sign-changing solutions of (\ref{PBN}) with exactly two nodal regions and such that $u_\lambda(0)>0$ for all $\lambda$. Assume that there exists $\lambda_0 \in \R^+$ such that $M_{\lambda,+} \rightarrow \infty$ as $\lambda \rightarrow \lambda_0$.
 Then $$M_{\lambda,-} \leq \frac{\lambda}{2} (1+o(1)),$$ for all $\lambda$ sufficiently close to $\lambda_0$.
\end{prop}

\begin{proof}
Let $(u_\lambda)$ be such a family of solutions. Since $N=6$, we have $p=\frac{N+2}{N-2}=2$ and thanks to the transformation (\ref{tran}) we have 
\begin{equation}\label{eq1prop3}
u_{\lambda}(r(t))=\lambda\ y(t;\gamma),
\end{equation}
for $t \in \left(\left(\frac{N-2}{\sqrt{\lambda}}\right)^{N-2},+\infty\right)$, where $\gamma=\lambda^{-1} M_{\lambda,+}$.
We observe that the global minimum point $s_{\lambda}$ corresponds, through the transformation (\ref{tran}), to  the number $t_0(\gamma)$ defined in \eqref{deftzero}. In fact by definition we have $u_{\lambda}^\prime(s_{\lambda})=0$ so it suffices to show that $u_{\lambda}^\prime(r) <0$ for all $r \in (0,s_{\lambda})$. By Corollary 1 in \cite{Iac} we know that $u_{\lambda}^\prime(r) <0$ for all $r \in (0,r_{\lambda})$, and for all $r \in (r_{\lambda},s_{\lambda})$. Moreover since $u_{\lambda}^+$ solves (\ref{PBN}) in $B_{r_{\lambda}}$, then, by Hopf lemma it follows that $u_{\lambda}^\prime(r_{\lambda})<0$. Now, thanks to the assumptions, as $\lambda \rightarrow \lambda_0$ we have $\gamma= \lambda^{-1}M_{\lambda,+} \rightarrow + \infty$  and the result follows immediately from (\ref{eq1prop3}) and Proposition \ref{proptran}.
\end{proof}

\begin{rem}
A straight important consequence of  Proposition \ref{propmlamb} is that $M_{\lambda,-}$ is uniformly bounded for all $\lambda$ sufficiently close to $\lambda_0$. In particular there cannot exist radial sign-changing solutions of (\ref{PBN}) with the shape of a tower of two bubbles in dimension $N=6$ (this fact also holds for the dimensions $N=3,4,5$, as we will see later). This is in deep contrast with the case of higher dimensions $N\geq 7$ as showed  in \cite{Iac}. 
\end{rem}

\begin{rem}\label{rem1}
In the case of the solutions obtained in the previous section, thanks to Theorem \ref{teolamb} we deduce that $M_{\lambda_2(\gamma),-}  \leq \frac{\lambda_0}{2}(1+o(1))\leq \frac{\lambda_1(B_1)}{2}$ for all sufficiently large $\gamma \in \R^+$.
\end{rem}

In the previous section we have studied the limit energy (see Theorem \ref{enestpp}) of the positive part of the solutions $u_{\lambda_2}$. Here we consider the negative part $u_{\lambda_2}^-$ and prove that its energy $J_{\lambda_2}$ is uniformly bounded as $\gamma \rightarrow +\infty$. This is the content of the next proposition.

\begin{prop} \label{prop4}
Let $\lambda_2=\lambda_2(\gamma)$ and $u_{\lambda_2}$ be the radial solution with exactly two nodal regions of (\ref{PBN})  described in Section 2. Let $J_\lambda (u):=\frac{1}{2} \left( \int_{B_1}|\nabla u|^2 - \lambda |u|^2 \ dx\right)-\frac{1}{2^*} \int_{B_1} |u|^{2^*} \ dx$ be the energy functional related to (\ref{PBN}).
Then $$J_{\lambda_2}(u_{\lambda_2}^-)\leq \frac{\pi^3}{36}\left(\frac{\lambda_1(B_1)}{2}\right)^3,$$
for all sufficiently large $\gamma$.
\end{prop}
\begin{proof}
 Since $u_{\lambda_2}^-$ solves $-\Delta u = \lambda_2 u + u^p$ in the annulus $A_{r_{\lambda_2}}$, in particular it belongs to the Nehari manifold $\mathcal{N}_{\lambda_2}$ associated to that equation, which is defined by 
 \begin{equation}\label{Neha}
 \mathcal{N}_{\lambda_2}:=\{u \in H_0^1(A_{r_{\lambda_2}}); \ \|u\|_{A_{r_{\lambda_2}}}^2 - \lambda_2 |u|_{2,A_{r_{\lambda_2}}}^2=|u|_{2^*,A_{r_{\lambda_2}}}^{2^*}\}.
 \end{equation}
Hence we deduce that 
\begin{equation}\label{stima1}
J_{\lambda_2}(u_{\lambda_2}^-)=\frac{1}{6}|u_{\lambda_2}^-|_{2^*,A_{r_{\lambda_2}}}^{2^*}. 
\end{equation}
Now, thanks to Proposition \ref{propmlamb}, (b) of Theorem \ref{teolamb} and Remark \ref{rem1} we have
\begin{equation} \label{stima2}
\begin{array}{lll}
\displaystyle |u_{\lambda_2}^-|_{2^*,A_{r_{\lambda_2}}}^{2^*} &=& \displaystyle \int_{A_{r_{\lambda_2}}}|u_{\lambda_2}^-|^{3} \ dx \\[14pt]
 &\leq & \displaystyle  |B_1| \|u_{\lambda_2}^-\|_\infty^3\\[8pt]
 &\leq & \displaystyle  \frac{\pi^3}{6}\left(\frac{\lambda_1(B_1)}{2}\right)^3,
\end{array}
\end{equation}
for all sufficiently large $\gamma$. From (\ref{stima1}) and (\ref{stima2}) we deduce the desired relation and the proof is complete.
\end{proof}

\begin{rem}\label{rem2}
Since $\lambda_2$ is a bounded function, by the same proof of Proposition \ref{prop4}, but without using (b) of Theorem, \ref{teolamb} we deduce anyway that $J_{\lambda_2}(u_{\lambda_2}^-)$ is uniformly bounded for all sufficiently large $\gamma$.
\end{rem}

We are interested now in studying the asymptotic behavior of the family $(u_{\lambda_2}^-)$. More precisely  we show that, as $\gamma \rightarrow \infty$, the family $(u_{\lambda_2}^-)$ converges in $C_{loc}^2(B_1-\{0\})$ to the unique positive solution $u_0$ of (\ref{PBN}) with $\lambda=\lambda_0$, for some $\lambda_0 \in (0,\lambda_1(B_1))$. We point out that these results will improve the energy estimate of $u_{\lambda_2}^-$ obtained before.

The pointwise convergence of $(u_{\lambda_2}^-)$ to $u_0$ is contained in Theorem 3 of \cite{ABP2}, but here we use a different approach which is based on the arguments of \cite{Iac}. Our result is the following:

\begin{teo} \label{teoriassnn6}
Let $N=6$, up to a subsequence, we have $\lambda_2(\gamma) \rightarrow \lambda_0$, as $\gamma \rightarrow +\infty$, for some $\lambda_0 \in (0,\lambda_1(B_1))$, and $(u_{\lambda_2}^-)$ converges in $C_{loc}^2(B_1-\{0\})$  to the unique solution $u_0$ of (\ref{PBNL0}). 
\end{teo}

\begin{proof}
%
Let us consider the family $(u_{\lambda_2}^-)$. These functions solve
\begin{equation}
\begin{cases}
-\Delta u = \lambda_2 u + u^p & \hbox{in } A_{r_{\lambda_2}},\\
u>0 & \hbox{in } A_{r_{\lambda_2}},\\
u=0 & \hbox{on } \partial A_{r_{\lambda_2}}.
\end{cases}
\end{equation}
Since $\lambda_2$ is bounded, up to a subsequence we have $\lim_{\gamma \rightarrow +\infty} \lambda_2 = \lambda_0$.
Thanks to Proposition \ref{propmlamb} we have that $u_{\lambda_2}^-$ is uniformly bounded for all sufficiently large $\gamma$ and by Lemma \ref{lem3} and the inverse transformation of (\ref{tran}) we have $r_{\lambda_2} \rightarrow 0$. Hence by standard elliptic theory, up to a subsequence, $u_{\lambda_2}^-$ converges in $C_{loc}^2(B_1-\{0\})$ as $\gamma \rightarrow + \infty$ to a solution $u_0$ of $-\Delta u = \lambda_0 u + u^p$ in $B_1-\{0\}$. 
We now proceed in four steps.
\\
\textbf{Step 1:} we have
\begin{equation}\label{lemdel1}
 \lim_{r\rightarrow 0}  u_0 (r)=\frac{\lambda_0}{2}.
\end{equation}
Since $u_{\lambda_2}^-$ is a radial solution of (\ref{PBN}) and thanks to Proposition \ref{propmlamb}, for all sufficiently large $\gamma$, we have 

\begin{equation}\label{eq1lemdel1}
u_{\lambda_2}^-\leq \frac{\lambda_0}{2}(1+o(1)), 
\end{equation}
and then we deduce that
\begin{equation*}
\begin{array}{lll}
[(u_{\lambda_2}^-)^\prime r^{5}]^\prime &=&-{\lambda_2} u_{{\lambda_2}}^-(r) r^{5}-[u_\lambda^-(r)]^{2}r^{5}\\[8pt]
&\geq&-{\lambda_2}\frac{\lambda_0}{2}(1+o(1)) r^{5}-\left[\frac{\lambda_0}{2}(1+o(1))\right]^2r^{5}\\[8pt]
&=&-\frac{\lambda_0^2}{2}(1+o(1))^2 r^{5}-\frac{\lambda_0^2}{4}(1+o(1))^2 r^{5}\\[8pt]
&\geq&-\lambda_0^2\ r^{5}.
\end{array}
\end{equation*}
Integrating between $s_{\lambda_2}$ and $r$ (with $s_{\lambda_2}<r<1$) we get that
\begin{equation*}
(u_{\lambda_2}^-)^\prime(r) r^{5} \geq\displaystyle -\lambda_0^2 \int_{s_{\lambda_2}}^{r}t^{5} dt
\geq-\frac{\lambda_0^2}{6} r^{6}.
\end{equation*}
Hence $(u_{\lambda_2}^-)^\prime(r) \geq -\frac{\lambda_0^2}{6} r$ for all  $r\in (s_{\lambda_2}, 1)$. Integrating again between $s_{\lambda_2}$ and $r$ we have
\begin{equation*}
 u_{\lambda_2}^-(r)-\frac{\lambda_0}{2}(1+o(1)) \geq -\frac{\lambda_0^2}{12} (r^{2}-s_{\lambda_2}^2)\geq-\frac{\lambda_0^2}{12} r^{2}.
\end{equation*}
Hence $u_{\lambda_2}^-(r) \geq \frac{\lambda_0}{2}(1+o(1))-\frac{\lambda_0^2}{12} r^{2}$ for all sufficiently large $\gamma$, for all $r \in (s_{\lambda_2}, 1)$. Since $s_{\lambda_2} \rightarrow 0$, then, passing to the limit as $\gamma \rightarrow \infty$, we get that  $ u_0(r) \geq  \frac{\lambda_0}{2}-\frac{\lambda_0^2}{12} r^{2}$, for all $0<r<1$. From this inequality and (\ref{eq1lemdel1}) we deduce that $\lim_{r\rightarrow 0} u_0(r)= \frac{\lambda_0}{2}$. The proof of Step 1 is complete.\\

\textbf{Step 2:} we have
\begin{equation}\label{lemdel2}
\lim_{r \rightarrow 0} u_0^\prime (r)=0.
\end{equation}
As in the previous step, integrating the equation between $s_{\lambda_2}$ and $r$, with $s_{\lambda_2}<r<1$,  we get that
\begin{equation*}
\begin{array}{lll}
-(\tilde u_{\lambda_2}^-)^\prime(r) r^{5} &=&\displaystyle {\lambda_2} \int_{s_{\lambda_2}}^{r} u_{\lambda_2}^- t^{5} dt + \int_{s_{\lambda_2}}^{r} ( u_{\lambda_2}^-)^{2} t^{5} dt .
\end{array}
\end{equation*}
Thanks to (\ref{eq1lemdel1}), for all sufficiently large $\gamma$ we have
\begin{equation*}
|( u_{\lambda_2}^-)^\prime(r) r^{5}| \leq \displaystyle {\lambda_2} \frac{\lambda_0}{2}(1+o(1)) \int_{s_{\lambda_2}}^{r} t^{5} dt + \frac{\lambda_0^2}{4}(1+o(1))^2 \int_{s_{\lambda_2}}^{r}t^{5} dt \leq \displaystyle \lambda_0^2 \frac{r^6}{6}.
\end{equation*}
Dividing by $r^5$ the previous inequality and passing to the limit, as $\gamma\rightarrow +\infty$, we get that  $$\displaystyle | u_0^\prime(r)|  \leq \frac{\lambda_0^2}{6} r,$$ for all $0<r<1$. Hence $\lim_{r\rightarrow 0} u_0^\prime(r)=0$ and the proof of Step 2 is complete.\\

From Step 1 and Step 2 it follows that the radial function $u_0(x)= u_0(|x|)$ can be extended to a $C^1(B_1)$ function. We still denote by $u_0$ this extension.\\ 
\textbf{Step 3:} The function $u_0$ is a weak solution in $B_1$ of 
\begin{eqnarray}\label{festproblim}
-\Delta u = \lambda_0 u + u^{2}. 
\end{eqnarray}
Let's fix a test function $\phi \in C_0^\infty (B_1)$. If $0 \notin supp(\phi)$ the proof is trivial so from now on we assume $0 \in  supp(\phi)$. Let $B_\delta$ be the ball  centered at the origin having radius $\delta>0$. 
Applying Green's formula to $\Omega(\delta):=B_1-B_\delta$, since $u_0$ is a $C^2(B_1-\{0\})$-solution of (\ref{festproblim}) and  $\phi\equiv 0$ on $\partial B_1$, we have

\begin{equation}\label{eq1fest}
 \int_{\Omega(\delta)} \nabla u_0 \cdot \nabla \phi \ dx = \lambda_0 \int_{\Omega(\delta)} \phi\ u_0 \ dx +\int_{\Omega(\delta)} \phi\ u_0^{2} \ dx + \int_{\partial B_\delta} \phi \left(\frac{\partial u_0}{\partial \nu}\right)\ d\sigma .
 \end{equation}
  We show now that $\int_{\partial B_\delta} \phi \left(\frac{\partial u_0}{\partial \nu}\right)\ d\sigma \rightarrow 0$ as $\delta \rightarrow 0$. In fact since $u_0$ is a radial function we have $\frac{\partial u_0}{\partial \nu}(x)= u_0^\prime (\delta) $ for all $x \in \partial B_\delta$, and hence we get that
\begin{eqnarray*}
\left| \int_{\partial B_\delta} \phi \left(\frac{\partial u_0}{\partial \nu}\right)\ d\sigma\right|
&\leq&| u_0^\prime (\delta)| \int_{\partial B_\delta} |\phi| \ d\sigma\\
&\leq& \omega_6| u_0^\prime (\delta)| \delta^{5} ||\phi||_{\infty}.
\end{eqnarray*}
Thanks to \eqref{lemdel2} we have $|u_0^\prime (\delta)| \delta^{5}\rightarrow 0$ as $\delta \rightarrow 0$.
To complete the proof we pass to the limit in (\ref{eq1fest}) as $\delta \rightarrow 0$. We observe that 
\begin{equation}\label{eq1wsol}
\begin{array}{lll}
\displaystyle |\nabla  u_0 \cdot \nabla \phi| \ \chi_{\Omega(\delta)} 
 &\leq&\displaystyle | \nabla u_0|^2 \ \chi_{\{| \nabla u_0|>1\}}|\nabla \phi| + | \nabla u_0| \ \chi_{\{| \nabla u_0|\leq1\}}|\nabla \phi|\\[10pt]
 &\leq&\displaystyle  | \nabla u_0|^2 \ \chi_{\{| \nabla u_0|>1\}}|\nabla \phi| +\  \chi_{\{| \nabla u_0|\leq1\}}|\nabla \phi|.
\end{array}
\end{equation}
We point out that  $\int_{B_1}|  \nabla u_0|^{2} dx$ is finite: this is an easy consequence of the fact that $u_{\lambda_2} \rightarrow u_0$ in $C^2_{loc}(B_1-\{0\})$, the family $(u_{\lambda_2})$ is uniformly bounded, \eqref{Neha} and Lebesgue's theorem.

Thus, since $\int_{B_1}|  \nabla u_0|^{2} dx$ is finite and $\phi$ has compact support, the right-hand side of (\ref{eq1wsol}) belongs to $L^1(B_1)$.
Hence from Lebesgue's theorem we have
\begin{equation}\label{eq2fest}
 \lim_{\delta \rightarrow 0} \int_{\Omega(\delta)} \nabla u_0 \cdot \nabla \phi \ dx = \int_{B_1} \nabla u_0 \cdot \nabla \phi \ dx.
\end{equation}
Since $\phi$ has compact support by Lebesgue's theorem we have 
\begin{equation}\label{eq3fest}
\begin{array}{lll}
\displaystyle \lim_{\delta \rightarrow 0}\int_{\Omega(\delta)} \phi\  u_0 \ dx &=& \displaystyle\int_{B_1} \phi\ u_0 \ dx,\\
\displaystyle \lim_{\delta \rightarrow 0}\int_{\Omega(\delta)} \phi\  u_0^{2} \ dx &=& \displaystyle\int_{B_1} \phi\  u_0^{2} \ dx.
\end{array}
\end{equation}
From  (\ref{eq1fest}), (\ref{eq2fest}), (\ref{eq3fest}) and since we have proved $\int_{\partial B(\delta)} \phi \left(\frac{\partial \tilde u}{\partial \nu}\right)\ d\sigma \rightarrow 0$ as $\delta \rightarrow 0$ it follows that
\begin{equation*}
 \int_{B_1} \nabla u_0 \cdot \nabla \phi \ dx =\lambda_0 \int_{B_1} \phi\  u_0 \ dx+\int_{B_1} \phi\  u_0^{2} \ dx,
 \end{equation*}
 which completes the proof of Step 3.\\
\textbf{Step 4:} It holds
\begin{equation} \label{lemdel3}
\lim_{r\rightarrow 1} u_0(r)=0.
\end{equation}

First we observe that since $u_{\lambda_2}^-$ is uniformly bounded, then, its energy and in particular its $H_0^1$-norm are uniformly bounded, for all sufficiently large $\gamma$ (see Proposition \ref{prop4} and its proof).  

Now, since $u_{\lambda_2}(1)=0$, we have
$$ u_{\lambda_2}^-(r) = - \int_r^1 (u_{\lambda_2}^-)^\prime(s) \ ds, $$
and hence
\begin{equation}\label{eq1lem9}
\begin{array}{lll}
|u_{\lambda_2}^-(r) | &\leq& \displaystyle \int_r^1 |(u_{\lambda_2}^-)^\prime(s)| \ ds\\[12pt]
 &\leq& \displaystyle \int_r^1 |(u_{\lambda_2}^-)^\prime(s)| \ ds\\[12pt]
  &\leq& \displaystyle \left(\int_r^1 |(u_{\lambda_2}^-)^\prime(s)|^2 \ ds\right)^{1/2} (1-r)^{1/2}.
\end{array}
\end{equation}
Since $|\nabla u_{\lambda_2}^-|_{2,A_{r_{\lambda_2}}}$ is uniformly bounded, then, taking the limsup in  (\ref{eq1lem9}) as $\gamma \rightarrow + \infty$ we get that
$$|u_{0}^-(r) | \leq c (1-r)^{1/2},$$
and from this we deduce that $ \lim_{r\rightarrow 1} u_{0}^-(r)=0$, and we are done. The proof of Step 4 is complete.\\

Thanks to Step 1 - Step 4 we get that $u_0 \in H_{0,rad}^1(B_1)$ is a weak solution of 
\begin{equation}\label{PBNL0}
\begin{cases}
-\Delta u = \lambda_0 u + u^2 & \hbox{in } B_1,\\
u>0 & \hbox{in } B_1,\\
u=0 & \hbox{on } \partial B_1.
\end{cases}
\end{equation}
In particular, as a consequence of a well known result of Brezis and Kato (for instance see  Lemma 1.5 in \cite{BN}) it is possible to show that  $u_0$ is a classical solution of (\ref{PBNL0}) (see Appendix B of \cite{STRUWE}). Thanks to \cite{AYU} we know that $u_0$ is the unique positive radial solution of (\ref{PBNL0}), which is the one found by Brezis and Nirenberg in \cite{BN}. Hence we must have $\lambda_0 < \lambda_1(B_1)$ and $J_{\lambda_0}(u_{\lambda_0}) < \frac{1}{6}S^3$.  
\end{proof}


Next result gives a characterization of the value $\lambda_0 \in (0,\lambda_1(B_1))$ appearing in Theorem \ref{teoriassnn6}.
\begin{teo}\label{carattlambda0}
Let $\lambda_0:=\lim_{\gamma \rightarrow +\infty} \lambda_2(\gamma)$. We have that $\lambda_0$ is the unique $\lambda \in (0,\lambda_1(B_1))$ such that $u_\lambda(0)=\frac{\lambda}{2}$, where $u_\lambda$ is the unique positive solution of 
\begin{equation}\label{PBNPL}
\begin{cases}
-\Delta u = \lambda u + u^2 & \hbox{in } B_1,\\
u>0 & \hbox{in } B_1,\\
u=0 & \hbox{on } \partial B_1.
\end{cases}
\end{equation}
\end{teo}

\begin{proof}
Thanks to Theorem \ref{teoriassnn6} and \eqref{lemdel1} we have that the set 
$$\Gamma:=\left\{\lambda \in (0,\lambda_1(B_1)); \ u_\lambda(0)=\frac{\lambda}{2}, \ \hbox{where} \ u_\lambda \ \hbox{is the unique solution of \eqref{PBNPL}}\right\}, $$
is not empty since $\lambda_0 \in \Gamma$. We want to prove that $\Gamma=\{\lambda_0\}$. To this end assume that $\bar\lambda \in \Gamma$ and $\bar \lambda \neq \lambda_0$. In particular the functions $u_{\lambda_0}$ and $u_{\bar\lambda}$ are different. Thanks to the definition of $\Gamma$ and applying \eqref{tran} (with $p=2$ because $N=6$) we get that $u_{\lambda_0}$ and $u_{\bar\lambda}$ are respectively transformed to a solution of \eqref{EFS} with $\gamma=\frac{1}{2}$, but, for a given $\gamma$, the solution of \eqref{EFS} is unique and this gives a contradiction.
\end{proof}

Now we have all the tools to estimate the energy of the solutions $u_{\lambda_2}$. This is the content of the next result.

\begin{cor}\label{enest}
 Let $N=6$ and let $u_{\lambda_2}$ be the radial solution with exactly two nodal regions of (\ref{PBN}) with $\lambda=\lambda_2(\gamma)$ obtained in Section 2. Then
$$J_{\lambda_2} (u_{\lambda_2}) < \frac{1}{3}S^{3},$$
for all sufficiently large $\gamma \in \R^+$, where $J_\lambda (u):=\frac{1}{2} \left( \int_{B_1}|\nabla u|^2 - \lambda |u|^2 \ dx\right)-\frac{1}{2^*} \int_{B_1} |u|^{2^*} \ dx$ is the energy functional related to (\ref{PBN}), $S$ is the best Sobolev constant for the embedding of $\mathcal{D}^{1,2}(\R^6)$ into $L^{2^*}(\R^6)$.
\end{cor}

\begin{proof}
Let $(u_{\lambda_2})$ be this family of solutions. Observe that $J_{\lambda_2} (u_{\lambda_2})=J_{\lambda_2} (u_{\lambda_2}^+)+J_{\lambda_2} (u_{\lambda_2}^-)$ hence it suffices to estimate separately the energy of the positive and negative part of $u_{\lambda_2}$. The energy of $u_{\lambda_2}^+$ has been determinated in Theorem \ref{enestpp}, and in particular we have $J_{\lambda_2}(u_{\lambda_2}^+) \rightarrow \frac{1}{6}S^3$, as $\gamma \rightarrow + \infty$.

%
%

Now we estimate $J_{\lambda_2}(u_{\lambda_2}^-)$. Since $u_{\lambda_2}^-$ solves $-\Delta u = \lambda_2 u + u^p$ in the annulus $A_{r_{\lambda_2}}$, in particular it belongs to the Nehari manifold $\mathcal{N}_{\lambda_2}$ associated to this equation, (see \eqref{Neha}).
Hence we deduce that $J_{\lambda_2}(u_{\lambda_2}^-)=\frac{1}{6}|u_{\lambda_2}^-|_{2^*,A_{r_{\lambda_2}}}^{2^*}$. To complete the proof it will suffice to show that $$|u_{\lambda_2}^-|_{2^*,A_{r_{\lambda_2}}}^{2^*} \rightarrow |u_{\lambda_0}^-|_{2^*,B_1}^{2^*},$$
where $u_0$ is the unique solution of (\ref{PBNL0}). In fact, thanks to Theorem \ref{teoriassnn6} we know that, up to a subsequence, $(u_{\lambda_2}^-)$ converges in $C_{loc}^2(B_1-\{0\})$  to the unique solution $u_0$ of (\ref{PBNL0}). Hence to prove our assertion it suffices to apply Lebesgue's theorem, which clearly holds since $(u_{\lambda_2}^-)$ is uniformly bounded as $\gamma \rightarrow +\infty$.

Now since $J_{\lambda_2}(u_{\lambda_2}^-) \rightarrow  J_{\lambda_0}(u_{\lambda_0})$ and  $J_{\lambda_0}(u_{\lambda_0}) < \frac{1}{6} S^3$ we deduce the desired relation.
\end{proof}

\section{Asymptotic analysis of the negative part in dimension $N=3,4,5$}
Here we prove:
\begin{teo} \label{propmlamb2}
Let $N=3,4,5$ and let $(u_\lambda)$ be any family of radial sign-changing solutions of (\ref{PBN}) with exactly two nodal regions and such that $u_\lambda(0)>0$ for all $\lambda$. Assume that there exists $\bar\lambda \in \R^+$ such that $M_{\lambda,+} \rightarrow \infty$, as $\lambda \rightarrow \bar\lambda$.
 Then 
 \begin{itemize}
\item[(i):] $M_{\lambda,-} \rightarrow 0$, as $\lambda \rightarrow \bar\lambda$;
\item[(ii):] $(u_{\lambda}^-)$ converges to zero uniformly in $B_1$, as $\lambda\rightarrow \bar\lambda$.
\end{itemize}
\end{teo}

\begin{proof} We start by proving (i).
Let $(u_\lambda)$ be such a family of solutions. Thanks to the transformation (\ref{tran}) we have 
\begin{equation}\label{eq1prop3N}
u_{\lambda}(r(t))=\lambda^{\frac{1}{p-1}}\ y(t;\gamma),
\end{equation}
for $t \in \left(\left(\frac{N-2}{\sqrt{\lambda}}\right)^{N-2},+\infty\right)$, where $\gamma=\lambda^{-\frac{1}{p-1}} M_{\lambda,+}$ and $y=y(t;\gamma)$ solves (\ref{EFS}). Clearly, as $\lambda \rightarrow \bar\lambda$, we have $\gamma\rightarrow+ \infty$. As in the proof of Proposition \ref{propmlamb} we have that the global minimum point $s_{\lambda}$ corresponds, through the transformation (\ref{tran}), to  the number $t_0(\gamma)$ defined in \eqref{deftzero}. 

Hence, thanks to Lemma \ref{lem2}, it holds
\begin{equation}\label{stimmintzero}
|y(t_0(\gamma);\gamma)| < |y^\prime(T_1(\gamma))| \ (T_1(\gamma)- t_0(\gamma)). 
\end{equation}
For $N=3$, which corresponds to $k=4$, by Lemma \ref{lem3} we have that $T_1(\gamma)$ is uniformly bounded for all sufficiently large $\gamma$, while, by Lemma \ref{lem4} it holds $y^\prime(T_1(\gamma))=(k-1)^{\frac{1}{k-2}}\gamma^{-1}(1+o(1))$. Thus, since $0<t_0(\gamma)<T_1(\gamma)$, from \eqref{stimmintzero}, \eqref{eq1prop3N} we get that $M_{\lambda,-}=\lambda^{\frac{1}{p-1}} y(t_0;\gamma) \rightarrow 0$ as $\lambda\rightarrow \bar\lambda$.

For $N=4$, which corresponds to $k=3$, by Lemma \ref{lem3} we have that $T_1(\gamma)=2 \log(\gamma) (1+o(1))$ for all sufficiently large $\gamma$, and hence as in the previous case, we get that $M_{\lambda,-}=\lambda^{\frac{1}{p-1}} y(t_0;\gamma) \rightarrow 0$ as $\lambda\rightarrow \bar\lambda$. The same happens for $N=5$ ($k=8/3$); in fact by Lemma \ref{lem3} we have that $T_1(\gamma)=A\gamma^{2/3} (1+o(1))$ for all sufficiently large $\gamma$, where $A=A(k)$ is a positive constant depending only on $k$ (see Lemma \ref{lem3} for its definition). The proof of (i) is complete.

Now we prove (ii). We recall that 
 $u_{\lambda}^-$ is nonzero in the annulus $A_{r_{\lambda}}(0)=\{x \in \R^N; \ r_{\lambda} < |x|<1\}$ and vanishes outside. Thanks  to (i), we have $\|u_{\lambda}\|_{\infty,B_1} =M_{\lambda,-}\rightarrow 0$ as $\lambda\rightarrow \bar\lambda$ and we are done.
\end{proof}
%

\end{document}